\documentclass[reqno]{amsart}

\usepackage{etoolbox}
\makeatletter
\let\old@tocline\@tocline
\let\section@tocline\@tocline

\newcommand{\subsection@dotsep}{4.5}
\newcommand{\subsubsection@dotsep}{4.5}
\patchcmd{\@tocline}
  {\hfil}
  {\nobreak
     \leaders\hbox{$\m@th
        \mkern \subsection@dotsep mu\hbox{.}\mkern \subsection@dotsep mu$}\hfill
     \nobreak}{}{}
\let\subsection@tocline\@tocline
\let\@tocline\old@tocline

\patchcmd{\@tocline}
  {\hfil}
  {\nobreak
     \leaders\hbox{$\m@th
        \mkern \subsubsection@dotsep mu\hbox{.}\mkern \subsubsection@dotsep mu$}\hfill
     \nobreak}{}{}
\let\subsubsection@tocline\@tocline
\let\@tocline\old@tocline

\let\old@l@subsection\l@subsection
\let\old@l@subsubsection\l@subsubsection

\def\@tocwriteb#1#2#3{%
  \begingroup
    \@xp\def\csname #2@tocline\endcsname##1##2##3##4##5##6{%
      \ifnum##1>\c@tocdepth
      \else \sbox\z@{##5\let\indentlabel\@tochangmeasure##6}\fi}%
    \csname l@#2\endcsname{#1{\csname#2name\endcsname}{\@secnumber}{}}%
  \endgroup
  \addcontentsline{toc}{#2}%
    {\protect#1{\csname#2name\endcsname}{\@secnumber}{#3}}}%

\newlength{\@tocsectionindent}
\newlength{\@tocsubsectionindent}
\newlength{\@tocsubsubsectionindent}
\newlength{\@tocsectionnumwidth}
\newlength{\@tocsubsectionnumwidth}
\newlength{\@tocsubsubsectionnumwidth}
\newcommand{\settocsectionnumwidth}[1]{\setlength{\@tocsectionnumwidth}{#1}}
\newcommand{\settocsubsectionnumwidth}[1]{\setlength{\@tocsubsectionnumwidth}{#1}}
\newcommand{\settocsubsubsectionnumwidth}[1]{\setlength{\@tocsubsubsectionnumwidth}{#1}}
\newcommand{\settocsectionindent}[1]{\setlength{\@tocsectionindent}{#1}}
\newcommand{\settocsubsectionindent}[1]{\setlength{\@tocsubsectionindent}{#1}}
\newcommand{\settocsubsubsectionindent}[1]{\setlength{\@tocsubsubsectionindent}{#1}}

\renewcommand{\l@section}{\section@tocline{1}{\@tocsectionvskip}{\@tocsectionindent}{\@tocsectionnumwidth}{\@tocsectionformat}}%
\renewcommand{\l@subsection}{\subsection@tocline{1}{\@tocsubsectionvskip}{\@tocsubsectionindent}{\@tocsubsectionnumwidth}{\@tocsubsectionformat}}%
\renewcommand{\l@subsubsection}{\subsubsection@tocline{1}{\@tocsubsubsectionvskip}{\@tocsubsubsectionindent}{\@tocsubsubsectionnumwidth}{\@tocsubsubsectionformat}}%
\newcommand{\@tocsectionformat}{}
\newcommand{\@tocsubsectionformat}{}
\newcommand{\@tocsubsubsectionformat}{}
\expandafter\def\csname toc@1format\endcsname{\@tocsectionformat}
\expandafter\def\csname toc@2format\endcsname{\@tocsubsectionformat}
\expandafter\def\csname toc@3format\endcsname{\@tocsubsubsectionformat}
\newcommand{\settocsectionformat}[1]{\renewcommand{\@tocsectionformat}{#1}}
\newcommand{\settocsubsectionformat}[1]{\renewcommand{\@tocsubsectionformat}{#1}}
\newcommand{\settocsubsubsectionformat}[1]{\renewcommand{\@tocsubsubsectionformat}{#1}}
\newlength{\@tocsectionvskip}
\newcommand{\settocsectionvskip}[1]{\setlength{\@tocsectionvskip}{#1}}
\newlength{\@tocsubsectionvskip}
\newcommand{\settocsubsectionvskip}[1]{\setlength{\@tocsubsectionvskip}{#1}}
\newlength{\@tocsubsubsectionvskip}
\newcommand{\settocsubsubsectionvskip}[1]{\setlength{\@tocsubsubsectionvskip}{#1}}

\patchcmd{\tocsection}{\indentlabel}{\makebox[\@tocsectionnumwidth][l]}{}{}
\patchcmd{\tocsubsection}{\indentlabel}{\makebox[\@tocsubsectionnumwidth][l]}{}{}
\patchcmd{\tocsubsubsection}{\indentlabel}{\makebox[\@tocsubsubsectionnumwidth][l]}{}{}

\newcommand{\@sectypepnumformat}{}
\renewcommand{\contentsline}[1]{%
  \expandafter\let\expandafter\@sectypepnumformat\csname @toc#1pnumformat\endcsname%
  \csname l@#1\endcsname}
\newcommand{\@tocsectionpnumformat}{}
\newcommand{\@tocsubsectionpnumformat}{}
\newcommand{\@tocsubsubsectionpnumformat}{}
\newcommand{\setsectionpnumformat}[1]{\renewcommand{\@tocsectionpnumformat}{#1}}
\newcommand{\setsubsectionpnumformat}[1]{\renewcommand{\@tocsubsectionpnumformat}{#1}}
\newcommand{\setsubsubsectionpnumformat}[1]{\renewcommand{\@tocsubsubsectionpnumformat}{#1}}
\renewcommand{\@tocpagenum}[1]{%
  \hfill {\mdseries\@sectypepnumformat #1}}

\let\oldappendix\appendix
\renewcommand{\appendix}{%
  \leavevmode\oldappendix%
  \addtocontents{toc}{%
    \protect\settowidth{\protect\@tocsectionnumwidth}{\protect\@tocsectionformat\sectionname\space}%
    \protect\addtolength{\protect\@tocsectionnumwidth}{2em}}%
}
\makeatother

\makeatletter
\settocsectionnumwidth{2em}
\settocsubsectionnumwidth{2.5em}
\settocsubsubsectionnumwidth{3em}
\settocsectionindent{1pc}%
\settocsubsectionindent{\dimexpr\@tocsectionindent+\@tocsectionnumwidth}%
\settocsubsubsectionindent{\dimexpr\@tocsubsectionindent+\@tocsubsectionnumwidth}%
\makeatother

\settocsectionvskip{2pt}
\settocsubsectionvskip{0pt}
\settocsubsubsectionvskip{0pt}

\settocsectionformat{\bfseries}
\settocsubsectionformat{\mdseries}
\settocsubsubsectionformat{\mdseries}
\setsectionpnumformat{\bfseries}
\setsubsectionpnumformat{\mdseries}
\setsubsubsectionpnumformat{\mdseries}

\let\oldtableofcontents\tableofcontents
\renewcommand{\tableofcontents}{%
  \vspace*{-5\linespacing}% Default gap to top of CONTENTS is \linespacing.
  \oldtableofcontents}

\makeatletter
\let\origsection=\section \def\section{\@ifstar{\origsection*}{\mysection}} 
\def\mysection{\@startsection{section}{1}\z@{.7\linespacing\@plus\linespacing}{.5\linespacing}{\normalfont\scshape\centering\S}}
\makeatother  

\usepackage{amssymb}
\usepackage{amsmath}
\usepackage{amsthm}
\usepackage{letterswitharrows}
\usepackage{mathrsfs}
\usepackage{graphicx}
\usepackage{latexsym}
\usepackage{stmaryrd}
\usepackage[shortlabels]{enumitem}
\usepackage{moreenum}
\usepackage[bookmarks,hyperfootnotes=false, psdextra=true]{hyperref}
\usepackage{multirow}
\usepackage{xcolor}
\usepackage{caption}
\usepackage{subcaption}

\colorlet{darkishRed}{red!60!black}
\colorlet{darkishBlue}{blue!60!black}
\colorlet{darkishGreen}{green!50!black}
\colorlet{darkerishGreen}{green!30!black}
\colorlet{lightishGreen}{green!70!black}
\hypersetup{
    draft = false,
    bookmarksopen=true,
    colorlinks,
    linkcolor={darkishBlue},
    citecolor={darkishBlue},
    urlcolor={darkishBlue}
}
\usepackage[nameinlink, capitalise, noabbrev]{cleveref}
\usepackage{thmtools}
\usepackage{nameref}
\crefformat{enumi}{#2#1#3}
\crefformat{equation}{#2(#1)#3}
\crefrangeformat{section}{Sections~#3#1#4--#5#2#6}
\crefname{mainresult}{Theorem}{Theorems}
\usepackage{doi}
\let\setminus=\smallsetminus
\usepackage{tikz}
\usepackage{tikz-cd}
\usetikzlibrary{calc,through,intersections,arrows, trees, positioning, decorations.pathmorphing, cd}
\usepackage{comment}
\usepackage{mathtools}
\usepackage{nccmath}
\usepackage{pifont}
\usepackage[utf8]{inputenc}
\usepackage[T1]{fontenc}
\usepackage{lmodern}
\usepackage[babel]{microtype}
\usepackage[english]{babel}
\usepackage{relsize}

\newcommand{\COMMENT}[1]{{}}

\usepackage{geometry}
\let\setminus=\smallsetminus
\renewcommand{\leq}{\leqslant}
\renewcommand{\geq}{\geqslant}
\renewcommand{\ge}{\geq}

\let\rho=\varrho
\let\phi=\varphi

\newcommand{\id}{\normalfont\text{id}}

\newcommand{ \N } { \mathbb{N} }
\newcommand{ \Z } { \mathbb{Z} }

\newcommand{\defn}[1]{{\color{darkishGreen}{\emph{#1}}}}

\makeatletter

\def\calCommandfactory#1{%
   \expandafter\def\csname c#1\endcsname{\mathcal{#1}}}
\def\frakCommandfactory#1{%
   \expandafter\def\csname frak#1\endcsname{\mathfrak{#1}}}

\newcounter{ctr}
\loop
  \stepcounter{ctr}
  \edef\X{\@Alph\c@ctr}
  \expandafter\calCommandfactory\X
  \expandafter\frakCommandfactory\X
  \edef\Y{\@alph\c@ctr}
  \expandafter\frakCommandfactory\Y
\ifnum\thectr<26
\repeat

\renewcommand{\cD}{\mathscr{D}}

\setenumerate{label={\normalfont (\roman*)}}%,itemsep=0pt}

\newtheorem{theorem}{Theorem}[section] 
\newtheorem{corollary}[theorem]{Corollary}
\newtheorem{lemma}[theorem]{Lemma}

\newtheorem{conjecture}[theorem]{Conjecture}

\crefname{claim}{Claim}{Claims}
\AtEndEnvironment{proof}{\setcounter{claim}{0}}

\newtheorem{innercustomthm}{}

\newenvironment{customthm}[1]
  {\renewcommand\theinnercustomthm{#1}\innercustomthm}
  {\endinnercustomthm}

\theoremstyle{definition}

\newtheorem{example}[theorem]{Example}

\newtheorem*{definition*}{Definition}

\newtheorem{question}[theorem]{Question}

\theoremstyle{remark}

\usepackage{etoolbox}
\newbool{pdfBool}
\boolfalse{pdfBool} % Set the boolean variable to false

\newcommand{\pdfOrNot}[2]{\ifbool{pdfBool}{{#1}}{{#2}}}

\usepackage{svg}

\def\lqedsymbol{\ifmmode$\lrcorner$\else{\unskip\nobreak\hfil
		\penalty50\hskip1em\null\nobreak\hfil$\rule{1.2ex}{1.2ex}$
		\parfillskip=0pt\finalhyphendemerits=0\endgraf}\fi}

\newcommand{\td}{tree-decom\-pos\-ition}

\usepackage{tikz}
\usetikzlibrary{positioning,shapes,calc}

\usepackage{nicefrac}
\usepackage{csquotes}
\usepackage[backend=biber, style=alphabetic, maxnames=99, maxalphanames=99, sortcites=true]{biblatex}
\DeclareLabelalphaTemplate{
    \labelelement{
        \field[final]{shorthand}
        \field{label}
        \field[strwidth=1,strside=left]{labelname}
    }
    \labelelement{
        \field[strwidth=2,strside=right]{year}
    }
}

\begin{document}

\title{Locally chordal graphs}

\author[T.\ Abrishami \and P.\ Knappe \and J.\ Kobler]{Tara Abrishami, Paul Knappe, Jonas Kobler}

\address{Stanford University, Department of Mathematics}
\email{tara.abrishami@stanford.edu}

\address{Universität Hamburg, Department of Mathematics}
\email{paul.knappe@uni-hamburg.de}

\thanks{T.A. was supported by the National Science Foundation Award Number DMS-2303251 and the Alexander von Humboldt Foundation. P.K. was supported by a doctoral scholarship of the Studienstiftung des deutschen Volkes.}


\begin{abstract}
In this paper we study locally chordal graphs, i.e.\ graphs where every small-radius ball is chordal.
We prove four characterizations of locally chordal graphs.
Two are counterparts of the classic descriptions of chordal graphs via induced subgraphs and via minimal separators.
For the latter, we rely on the local separators introduced in \cite{computelocalSeps}.
Another characterization is via the local covering, which was introduced in \cite{canonicalGD} to study local-global characteristics of graphs using coverings from topology.
Our final characterization of locally chordal graphs is in terms of their binary cycle spaces. 
This gives a new characterization of chordal graphs as wheel-free graphs whose binary cycle space is generated by triangles.

Together, these results demonstrate the potential of local-global tools to uncover rich new properties. Our results in this paper also form the basis of our local-global analysis of locally chordal graphs \cite{localGlobalChordal}, where we develop a local-global perspective into structural characterizations.
\end{abstract}

\maketitle

\section{Introduction}
Chordal graphs are highly structured graphs which admit many equivalent characterizations, but their strong structure is a double-edged sword.
On one hand, their structure can be used to prove a variety of descriptive characterizations for chordal graphs, making them a well-known and fundamental graph class.
On the other hand, they are so structured that it is difficult for them to contribute to the forefront of new or inventive ideas in graph theory.  Our goal in this paper is to define a new class of graphs, \emph{locally chordal graphs}, which occupy the best of both worlds: they retain some structure that can be effectively leveraged, but they are broad enough to offer insight on the emerging research frontier of (structural) local-global studies.  

A graph is \defn{chordal}\footnote{In the literature, chordal graphs are sometimes also called \emph{rigid circuit} or \emph{triangulated} graphs.} if each of its cycles of length at least four has a chord. 
Let $r \geq 0$ be an integer, the given \emph{degree of locality}.
For a vertex $v$ of a graph $G$, we define the \defn{ball of radius $r/2$ around $v$}, denoted \defn{$B_{G}(v,r/2)$} or, for short, \defn{$B_{r/2}(v)$}, as the subgraph of $G$ given by the set of all vertices of distance at most $r/2$ from $v$ and the set of all edges of $G$ for which the sum of the endpoints' distances to $v$ is strictly less than $r$.\footnote{See \cref{sec:chordallocchordal} for a comparison to the \emph{depth-$d$ closed neighborhood}.}
In this paper, we study \defn{$r$-locally chordal} graphs: graphs whose balls of radius $r/2$ are chordal.
Our main result is that several of the well-known characterizing properties of chordal graphs are valid at any scale of locality, rather than just at their global extreme: 

\begin{restatable}{mainresult}{mainthm}\label{BasicCharacterization}
    Let $G$ be a (possibly infinite) graph and let $r \geq 3$ be an integer. The following are equivalent:
    \begin{enumerate}
        \item \label{basic:i} $G$ is $r$-locally chordal.
        \item\label{basic:ii} $G$ is $r$-chordal and wheel-free.
        \item \label{basic:iii} The $r$-local cover $G_r$ of $G$ is chordal.
        \item\label{basic:iv} Every minimal $r$-local separator of $G$ is a clique. 
    \end{enumerate}
\end{restatable}

These characterizations form the basis of our structural characterizations of locally chordal graphs in \cite{localGlobalChordal}, where we generalize the characterization of chordal graphs by \td s into cliques (equivalently: subtree intersection graphs).

Next, we explain \cref{basic:i}, \cref{basic:ii}, \cref{basic:iii}, and \cref{basic:iv}, and how they relate to descriptions of chordal graphs, in more detail.

\subsection{When the degree of locality goes to infinity}
Since trees are chordal and in every graph every ball of radius $r/2$ for $r \leq 2$ is a tree (specifically, a star), every graph is $0$-, $1$-, and $2$-locally chordal. From the definition, every graph which is $r$-locally chordal is $r'$-locally chordal for all $0 \leq r' \leq r$. Moreover, the graphs which are $r$-locally chordal for every $r \in \N$ are exactly the chordal graphs. Therefore, the sequence of $r$-locally chordal graphs for $r \in \N$ is an inclusion-wise decreasing sequence of graph classes, starting from the class of all graphs and ending with the class of chordal graphs as its limit. 
Next we discuss the characterizations \cref{basic:ii}, \cref{basic:iii} and \cref{basic:iv} of $r$-locally chordal graphs for $r \geq 3$.

\subsection{Characterization via induced subgraphs} 
The definition of chordal graphs can be easily understood via induced subgraphs: a graph is chordal if and only if it does not contain an induced cycle of length at least four. A graph is \defn{$r$-chordal} if each of its cycles of length at least 4 and at most $r$ has a chord. Every cycle of length at most $k \leq r$ in $G$ appears in $B_G(v, r/2)$ for each of its vertices $v$, so every $r$-locally chordal graph is $r$-chordal. 
A natural first guess is that the converse holds as well: that a graph is $r$-locally chordal if and only if it is $r$-chordal. 

However, long induced cycles can appear in a small-radius ball around a vertex even in an $r$-chordal graph. The simplest examples for this are wheels.
A \defn{wheel $W_n$} for $n \geq 4$ consists of a cycle of length $n$, its \emph{rim}, and a vertex $v$ complete to the rim, its \emph{hub}.
Every wheel is contained in the 3/2-ball around its hub, so every graph that contains a wheel as an induced subgraph is not $r$-locally chordal for any $r \geq 3$.
Therefore, every $r$-locally chordal graph is not only $r$-chordal but also \defn{wheel-free}, i.e.\ it does not contain a wheel $W_n$ with $n \geq 4$ as an induced subgraph. 
Surprisingly, \cref{BasicCharacterization}~\cref{basic:i}$\leftrightarrow$\cref{basic:ii} yields that these two necessary conditions are in fact sufficient: a graph is $r$-locally chordal if and only if it is $r$-chordal and wheel-free. 

\subsection{Characterization via coverings}

Recently, Diestel, Jacobs, Knappe and Kurkofka~\cite{canonicalGD} introduced a new framework to discriminate between the local structure of a graph and its global structure. 
Given a (possibly infinite) graph $G$ and a parameter $r \geq 0$, they construct a (possibly infinite) graph $G_r$, called the {\em $r$-local cover of $G$}, that represents only the local structure of $G$. 
Note that even if $G$ is finite, $G_r$ might be infinite.
The \emph{$r$-local covering $p_r \colon G_r \to G$} is defined in \cite[\S 4]{canonicalGD} via the characteristic subgroup of the fundamental group of $G$. 
Equivalently, one may describe the $r$-local covering of $G$ as the covering of $G$ that is universal among those that preserve the $r/2$-balls \cite[Lemma 4.2, 4.3 \& 4.4]{canonicalGD} 
(see \cref{sec:background:loc-cov} for a formal definition).

Since $p_r$ preserves $r/2$-balls, every $r/2$-ball of $G$ appears as an $r/2$-ball of $G_r$ (possibly more than once). 
Therefore, if the $r$-local cover $G_r$ of a graph $G$ is chordal, then $G$ is $r$-locally chordal.  
\cref{BasicCharacterization}~\cref{basic:i}$\leftrightarrow$\cref{basic:iii} yields that the converse is also true: 
for every integer $r \geq 0$,\footnote{For $r = 0,1,2$, this is trivially true, as not only is every graph is $r$-locally chordal, but it is also easy to see that the corresponding $r$-local cover $G_r$ is the universal cover of $G$, which is well-known to be a forest and thus chordal.} a graph $G$ is $r$-locally chordal if and only if $G_r$ is chordal.
This characterization ties locally chordal graphs to their local cover. 
Our results can thus provide a possible roadmap for studying locally structured graphs via local coverings in other settings.

\subsection{Algebraic characterization}
Our study of locally chordal graphs relies on the $r$-local covering, introduced in \cite{canonicalGD}, to represent the local structure of a graph.
This covering opens up a perspective on locally chordal graphs using tools from algebraic topology.
Following this, we find an algebraic characterization of locally chordal graphs in the class of wheel-free graphs via properties of their binary cycle spaces, i.e.\ their first homology group with coefficients in $\Z / {2\Z}$.
For readers not familiar with the binary cycle space, we include an equivalent formulation using symmetric differences of the edge sets of certain cycles.

\begin{restatable}{mainresult}{algLocChordal}\label{thm:chordalIFF-CSgenbyTriangles-No-Induced-wheels}
    Let $G$ be a (possibly infinite) wheel-free graph and let $r$ be a positive integer. 
    Then $G$ is $r$-chordal if and only if the cycles of $G$ of length at most $r$ in $G$ are generated by the triangles in the binary cycle space of $G$ (equivalently: each cycle of $G$ of length at most $r$ is the symmetric difference of finitely many triangles).
\end{restatable}

Since a graph is chordal if and only if it is wheel-free and $r$-chordal for every $r \in \N$,
an immediate consequence is an algebraic characterization of chordal graphs, which to our knowledge was previously unkown:

\begin{restatable}{mainresult}{algChordal}\label{cor:char-chordal-via-wheelfree-cyclespace}
A (possibly infinite) graph $G$ is chordal if and only if $G$ is wheel-free and its binary cycle space is generated by its triangles (equivalently: each of its cycles is a symmetric difference of finitely many triangles). 
\end{restatable}

\subsection{Characterization via minimal local separators}
\emph{Local components} at a vertex-subset $X$ of a graph $G$ were introduced in \cite{computelocalSeps} as a local analog to \emph{componental cuts}, the sets of edges between a component of $G-X$ and $X$.
Based on this notion, they introduced \emph{$r$-local separators}, a local analog of separators.
Intuitively, a local separator is a vertex set whose incident edges can be partitioned in at least two classes so that there is no {\em short} walk from one class to another (see \cref{sec:CharViaMinLocSep} for the formal definition). 
When $r$ goes to infinity, these become the usual separators.

Dirac~\cite[Theorem~1]{dirac1961rigid} characterized chordal graphs by the structure of their minimal separators: a graph is chordal if and only if all of its minimal separators are cliques. In this paper, we introduce \emph{minimal local separators} as a local analog to minimal separators.
\cref{BasicCharacterization}~\cref{basic:i}$\leftrightarrow$\cref{basic:iv} yields that locally chordal graphs are precisely the graphs whose minimal local separators are all cliques. 

\subsection{How this paper is organized}

In \cref{sec:CharactizationIndSubCov}, we prove the characterizations of locally chordal graphs via induced subgraphs and via the $r$-local cover, i.e.\ that statements \cref{basic:i}, \cref{basic:ii} and \cref{basic:iii} in \cref{BasicCharacterization} are equivalent.
The characterizations in terms of the binary cycle space, i.e.\ \cref{thm:chordalIFF-CSgenbyTriangles-No-Induced-wheels} and \cref{cor:char-chordal-via-wheelfree-cyclespace}, are proven in \cref{subsec:binarycyclespace}.
Then we prove in \cref{sec:CharViaMinLocSep} the characterization via minimal local separators, i.e.\ the equivalence of statements \cref{basic:i} and \cref{basic:iv} of \cref{BasicCharacterization}.
The background required for each proof is included in the subsections immediately preceding the proof.

\subsection{Acknowledgements}
We thank Raphael W. Jacobs for his involvement in the early stages of this project, and also for his feedback on a draft of this paper. 
We thank Jacob Stegemann for his input to the proof of \cref{thm:chordalIFF-CSgenbyTriangles-No-Induced-wheels}. 
We thank Sandra Albrechtsen for her first proof of \cref{BasicCharacterization}~\cref{basic:i}$\to$\cref{basic:iv} via the results from \cite{localGlobalChordal}. 
We thank Reinhard Diestel for his thorough input regarding the presentation of the results.

\subsection{Notation \texorpdfstring{\&}{and} conventions}

In this paper we deal with simple graphs. 
We specify when we deal with finite graphs or with general (possibly infinite) graphs. 
We follow the graph theory notation and definitions from \cite{bibel} and the topology terminology from \cite{Hatcher}. 

\section{Characterization via induced subgraphs and via coverings}\label{sec:CharactizationIndSubCov}

In this section, we prove the characterizations of locally chordal graphs via induced subgraphs and via the $r$-local cover, i.e.\ that statements \cref{basic:i}, \cref{basic:ii} and \cref{basic:iii} in \cref{BasicCharacterization} are equivalent, which we restate here for convenience:

\begin{customthm}{\cref*{BasicCharacterization}}\label{thm:FIRSTTHREE}
        Let $G$ be a (possibly infinite) graph and let $r \geq 3$ be an integer. The following are equivalent:
    \begin{enumerate}
        \item \label{basic:i:sec} $G$ is $r$-locally chordal.
        \item\label{basic:ii:sec} $G$ is $r$-chordal and wheel-free.
        \item \label{basic:iii:sec} The $r$-local cover $G_r$ of $G$ is chordal.
    \end{enumerate}
\end{customthm}

We now prove \hyperref[proof:FIRSTTHREE:iToii]{\cref*{thm:FIRSTTHREE}~\cref*{basic:i:sec}$\rightarrow$\cref*{basic:ii:sec}}. 
Then we complete the proof of the equivalences of \cref{thm:FIRSTTHREE}~\cref{basic:i:sec}, \cref{basic:ii:sec} and \cref{basic:iii:sec} in the remainder of this section: we show \hyperref[proof:FIRSTTHREE:iiiToi]{\cref*{basic:iii:sec}$\rightarrow$\cref*{basic:i:sec}} in \cref{sec:chordallocchordal} and \hyperref[proof:FIRSTTHREE:iiToiii]{\cref*{basic:ii:sec}$\rightarrow$\cref*{basic:iii:sec}} in \cref{subsec:binarycyclespace}.

\begin{proof}[Proof of \cref{thm:FIRSTTHREE}~\cref{basic:i:sec}$\rightarrow$\cref{basic:ii:sec}]
    \phantomsection\label{proof:FIRSTTHREE:iToii}
    Assume that \cref{basic:i:sec} holds.
    Any cycle in $G$ of length at most $r$ is contained in the $r/2$-ball around any of its vertices and any wheel in $G$ is contained in the $3/2$-ball around its hub.
    Since $r \geq 3$ and the $r/2$-balls of $G$ are chordal by~\cref{basic:i}, neither a cycle of length at most $r$ nor a wheel $W_n$ with $n \geq 4$ can be an induced subgraph of $G$. 
    Thus, $G$ is $r$-chordal and wheel-free, i.e.\ \cref{basic:ii:sec} holds.
\end{proof}

\subsection{Background: Minimal separators}\label{subsec:separators}
Let $u,v$ be two vertices and $X$ a vertex-subset of a graph $G$.
We say that $X$ \defn{separates $u$ and $v$} in $G$, and call $X$ a \defn{$u$--$v$ separator} of $G$, if $X \subseteq V(G) \setminus \{u,v\}$ and every $u$--$v$ path in $G$ meets $X$.
If $X$ separates some two vertices in $G$, then~$X$ is a \defn{separator} of $G$.
A separator $X$ is a \defn{minimal (vertex) separator} of $G$ if there are two vertices $u, v$ of $G$ such that $X$ is an inclusion-minimal $u$--$v$ separator in $G$. 

Dirac characterized chordal graphs  by the structure of their minimal separators: 

\begin{restatable}[Dirac~\cite{dirac1961rigid}, {Theorem~1}]{theorem}{tightsepchar}
    \label{thm:CharacterisationChordalViaMinimalVertexSeparator} 
    A (possibly infinite)\footnote{It is immediate from the proof presented in \cite[Theorem~2.1]{blair1993introduction} that this also holds for infinite graphs.} graph is chordal if and only if  all its minimal separators are cliques.
\end{restatable} 

\subsection{Chordal graphs are locally chordal}\label{sec:chordallocchordal}

In this section, we show that chordal graphs are $r$-locally chordal for any integer $r \geq 3$. More generally, we prove that the balls around connected vertex-subsets in chordal graphs are chordal.
Given a vertex-subset $X$ in a graph $G$ and an integer $r \geq 0$, 
the \defn{ball of radius $r/2$ around $X$} in $G$, denoted by \defn{$B_G(X,r/2)$} or, for short, \defn{$B_{r/2}(X)$}, is $\bigcup_{v \in X} B_{r/2}(v)$.

\begin{theorem}\label{lem:ChordalIsLocallyChordalForCliques}
    Let $G$ be a (possibly infinite) chordal graph and let $r \ge 3$ be an integer.
    If $G$ is chordal, then~$B_{r/2}(X)$ is chordal for any connected vertex-subset $X$ of $G$.
\end{theorem}

We recall some standard terminology.
Let $G$ be a graph, and let $X$ and $Y$ be two vertex-subsets of $G$.
An \defn{$X$--$Y$ path} is a path whose first vertex is in $X$, last vertex is in $Y$, and which is otherwise disjoint from $X \cup Y$.
An \defn{$X$-path} is a non-trivial path whose first and last vertex is in $X$ and which is otherwise disjoint from $X$.
A path $P$ is \defn{through $X$} if all internal vertices of $P$ are in $X$ and $P$ has at least one internal vertex. 

Let $d \geq 0$ be a real number.
The \defn{depth-$d$ closed neighborhood $N_G^d[X]$} is the set of all vertices of $G$ that have distance at most $d$ to $X$, and the \defn{depth-$d$ (open) neighborhood $N_G^d(X)$} is the set of all vertices of $G$ that have distance precisely $d$ to $X$. 
If the graph $G$ is clear from context, we also write \defn{$N^d[X]$} and \defn{$N^d(X)$}.
Note that the depth-$d$ open neighborhoods are always empty when $d$ is not an integer.
For any vertex $v$ of $G$, we use \defn{$N_G^d[v]$} and \defn{$N_G^d(v)$} as a shorthand for $N_G^d[\{v\}]$ and $N_G^d(\{v\})$. 

Let us now compare the depth-$d$ closed neighborhoods with the balls of radius $r$.
If $r$ is even, i.e.\ $r = 2d$ for $d \in \N$, then $B_{r/2}(v)$ is obtained from $G[N^d[v]]$ by deleting the edges between vertices of distance precisely $r/2$ to $v$.
Otherwise, if $r$ is odd, i.e.\ $r=2d+1$ for some $d \in \N$, then $B_{r/2}(v)$ is the induced subgraph  $G[N^d[v]]$ of $G$.

Moreover, if $r \geq 3$ is odd, then the $r/2$-ball around any connected vertex-subset $X$ in a chordal graph $G$ is the induced subgraph $G[N^d[X]]$ of $G$, and thus chordal.

\begin{lemma}\label{lem:RoddRadiusBallsAreInducedForCliques}
    Let $r \geq 3$ be an odd integer.
    Let $X$ be a connected vertex-subset of a (possibly infinite) chordal graph $G$.
    Then $B_{r/2}(X)$ is an induced subgraph of $G$.
\end{lemma}

\begin{proof}
    Let $d := \lfloor r/2 \rfloor$. Every edge $uv$ with one endvertex of distance strictly less than $d$ from $X$ is in $B_{r/2}(X)$. Thus,
    consider an edge $e = uv$ of $G$ whose endvertices $u,v$ are both in $B_{r/2}(X)$ and both have distance exactly $d$ from $X$.
    In particular, there are $x,y \in X$ such that $u \in N^d(x)$ and $v \in N^d(y)$.
    Fix a shortest $x$--$u$ path $P$ and a shortest $y$--$v$ path $Q$ in $G$.

    First, suppose that there is an edge $pq$ from a vertex $p$ of $P$ and a vertex $q$ of $Q$. Up to symmetry, we may assume that $d(x, p) \leq d(y, q)$. Suppose that $d(x, p) < d(y, q)$. Then there is a path from $x$ to $v$ of distance at most $d$, by concatenating the subpath of $P$ from $x$ to $p$, the edge $pq$, and the subpath of $Q$ from $q$ to $v$, so $
    \{u, v\} \subseteq B_{r/2}(x)$. Since $r$ is odd, the ball $B_{r/2}(x)$ is an induced subgraph of $G$, and so edge $uv \in B_{r/2}(x) \subseteq B_{r/2}(X)$. Therefore, we assume from now on that $P$ and $Q$ are disjoint and every edge between $P$ and $Q$ joins two vertices whose distance from $X$ are equal. 
    
    Let $R$ be a shortest $x$--$y$ path in the connected subgraph $G[X]$ of $G$.
    Now $R$, $P$, $Q$ and $uv$ form a cycle $O$.
    It can only have two types of chords $f$:
    Either $f$ is an edge that joins a vertex of $R$ and to the second vertex of $P$ or $Q$, or $f$ is an edge that joins a vertex of $P$ and a vertex of $Q$ whose distances from $X$ are the same.
    In the former case, we may shorten $R$, rename the endvertex of $f$ in $X$ to $x$ or $y$, and replace the first edge of $P$ or $Q$, respectively, by $f$.
    Since the new $R$ is shorter, we will thus eventually end up in the latter case.
    In the latter case, we choose $f$ to be such an edge whose endvertices have the largest possible distance from $X$ but less than $d$.
    Then the four paths $P$, $uv$, $Q$ and $f$ together contain an induced cycle of length at least $4$ in $G$.
    This contradicts the assumption that $G$ is chordal.
\end{proof}

\cref{lem:RoddRadiusBallsAreInducedForCliques} shows \cref{lem:ChordalIsLocallyChordalForCliques} for odd $r \geq 3$.
To prove \cref{lem:ChordalIsLocallyChordalForCliques} for even $r \geq 4$, we provide a lemma to find certain edges in $r/2$-balls of chordal graphs, which indeed will be how we find our desired chords in the proof of \cref{lem:ChordalIsLocallyChordalForCliques}:

\begin{lemma}\label{lem:edge-on-pre-last-levelForConnectedX}
    Let $r \geq 4$ be an integer and set $d: = \lfloor r/2 \rfloor$.
    Let $X$ be a connected vertex-subset of a (possibly infinite) chordal graph $G$.
    If, for any two distinct vertices $u,w \in N^{d-1}(X)$, there is a $u$--$w$ path in $B_{r/2}(X)$ through $N^d(X)$, then $uw$ is an edge of $B_{r/2}(X)$.
\end{lemma}

\begin{proof}

    Let $P$ be a $u$--$w$ path $P$ in $B_{r/2}(X)$ through $N^d(X)$.
    In particular, an internal vertex $z$ of $P$ is in $N^d(X)$.
    Note that, since $d \geq 2$, $N^{d-1}(X)$ is an $X$--$z$ separator in $G$.
    Let $S \subseteq N^{d-1}(v)$ be an inclusion-minimal $X$--$z$ separator.
    As the $u$--$w$ path $P$ through $N^d(X)$ contains $z$ and $S \subseteq N^{d-1}(X)$ is an $X$--$z$ separator in $G$, the endvertices $u,w$ of $P$ are in $S$.
    As $X$ is connected in $G$, $S$ is indeed a minimal $x$--$z$ separator for any $x \in X$.
    Thus, by \cref{thm:CharacterisationChordalViaMinimalVertexSeparator}, the minimal separator $S$ of the chordal graph $G$ is a clique.
    In particular, $uw$ is an edge of $G$, and hence of $B_{r/2}(X)$, as $u$ (and also $w$) is in $N^{d-1}(X)$.
\end{proof}

\begin{proof}[Proof of \cref{lem:ChordalIsLocallyChordalForCliques}]
    Let $X$ be a connected vertex-subset of $G$.
    We want to show that the $r/2$-ball $B := B_{r/2}(X)$ around $X$ in $G$ is chordal.
    If $r$ is odd,
    then $B$ is an induced subgraph by \cref{lem:RoddRadiusBallsAreInducedForCliques}, and thus $B$ is itself chordal.
    
    So assume that $r$ is even; in particular, $r \geq 4$.
    Let $O$ be a cycle of length at least $4$ in an $r/2$-ball $B$ around a connected vertex-subset $X$ in $G$.
    It suffices to show that $O$ has a chord.
    Since $r-1$ is odd, we have seen in the previous paragraph that the $(r-1)/2$-ball around~$X$ in~$G$ is chordal.
    Thus, $O$ has a chord or contains a vertex $z$ of distance $r/2 =: d$ from $X$.
    
    If $O$ contains a chord, then we are done.
    So we may assume that $O$ contains a vertex $z$ of distance $d$ to $X$.
    Let $u,w$ be the two (distinct) neighbors of $z$ on the cycle $O$.
    Since $r$ is even, it follows that $N^d(X)$ is independent in $B$, so every neighbor of $z$ in $B$ is in $N^{d-1}(X)$. 
    In particular, $u, w \in N^{d-1}(X)$.
    As $uzw$ is a $u$--$w$ path through $N^d(X)$, we obtain from \cref{lem:edge-on-pre-last-levelForConnectedX} that $uw \in B$.
    Hence, $uw$ is a chord of $O$ in $B$, as desired.
\end{proof}

\subsection{Background: Local cover \texorpdfstring{\&}{and} cliques}\label{sec:background:loc-cov}
 
We remind the reader that the graphs we consider in this paper are always simple, i.e.\ they neither have loops nor parallel edges.
A \defn{covering} of a (loopless) graph $G$ is a surjective homomorphism $p \colon \hat G \to G$ such that $p$ restricts to an isomorphism from the edges incident to any given vertex $\hat v$ of $\hat G$ to the edges incident to its projection $v \coloneqq p_r(\hat v)$. 
In this paper, we restrict our view to those coverings $p$ of graphs $G$ whose preimage $p^{-1}(C)$ of any component $C$ of $G$ is connected. Most of the time, this means we consider coverings $p: \hat G \to G$ where both $G$ and $\hat G$ are connected. 

Let $r \in \N$.
A (graph) homomorphism $p \colon \hat G \to G$ is \defn{$r/2$-ball-preserving} if $p$ restricts to an isomorphism from $B_{\hat G}(\hat v, r/2)$ to $B_{G}(v, r/2)$ for every vertex $\hat v$ of $\hat G$ and $v \coloneqq p_r(\hat v)$.
Thus, a surjective homomorphism $p \colon \hat G \to G$ is a covering of a (simple) graph $G$ if and only if $p$ is $2/2$-ball preserving.

In \cite{canonicalGD} the \defn{$r$-local covering $p_r \colon G_r \to G$} of a connected graph $G$ is introduced as the covering of $G$ whose characteristic subgroup is the $r$-local subgroup of the fundamental group of $G$.
Since the formal definition is not relevant to this paper, we refer the reader to \cite[\S 4]{canonicalGD} for details.
In this paper, we will use an equivalent characterization of the $r$-local covering.
The $r$-local covering is the universal $r/2$-ball-preserving covering of $G$ \cite[Lemma 4.2, 4.3 \& 4.4]{canonicalGD}: for every $r/2$-ball-preserving covering $p \colon \hat G \to G$ there exists a covering $q \colon G_r \to \hat G$ such that $p_r = p \circ q$.
Following this equivalent description of the $r$-local covering, it is also defined for non-connected graphs $G$.

We refer to the graph $G_r$ as the \defn{$r$-local cover} of $G$.
Note that the $0$-, $1$- and $2$-local covers of a (simple) graph $G$ are forests. For the reader who is familiar with coverings, we remark that the $0$-, $1$- and $2$-local coverings are all indeed the universal covering of $G$.
A \defn{deck transformation} of a covering $p \colon \hat G \to G$ is an automorphism $\gamma$ of $\hat G$ that commutes with $p$, i.e.\ $p = p \circ \gamma$.
We denote the group of deck transformations of a covering $p$ by \defn{$\Gamma(p)$}.

\medskip

A \defn{clique} $X$ in $G$ is a vertex-subset which induces a complete subgraph of $G$.
Local covers of graphs interact well with cliques:

\begin{lemma}[\cite{computelocalSeps}, Lemma~5.15]\label{lem:CliqueAndLocalCover}
    Let $G$ be a (possibly infinite) graph and let $r \geq 3$ an integer.
    \begin{enumerate}
        \item For every clique $\hat X$ of $G_r$, its projection $p_r(\hat X)$ is a clique of $G$, and $p_r$ restricts to a bijection from $\hat X$ to $p_r(\hat X)$.
        \item For every clique $X$ of $G$, there exists a clique $\hat X$ of $G_r$ such that $p_r$ restricts to a bijection from $\hat X$ to $X$.
        \item\label{item:CliquesMove} $N_{G_r}[\hat X] \cap N_{G_r}[\gamma(\hat X)] = \emptyset$ for every clique $\hat X$ of $G_r$ and every $\gamma \in \Gamma(p_r) \setminus \{\id_{G_r}\}$.
    \end{enumerate}
\end{lemma}

Given a covering $q \colon \hat G \to G$ of a graph $G$, a \defn{lift} of a vertex-subset $X$ of $G$ to $\hat G$ is a vertex-subset $\hat X$ of $\hat G$ that contains precisely one vertex of each fiber $q^{-1}(v)$ with $v \in X$.
Let us now consider the $r$-local covering $p_r \colon G_r \to G$ for any integer $r \geq 0$.
A vertex-subset $\hat X$ of $G_r$ with $X := p_r(\hat X)$ is \defn{$r$-locally closed} if no $r$-local $p_r^{-1}(X)$-walk starts with a vertex of $\hat X$ and ends in a vertex of $p_r^{-1}(X) \setminus \hat X$.

For a given clique $X$ of $G$, we refer to all cliques $\hat X$ of $G_r$ such that $p_r$ restricts to an isomorphism from $\hat X$ to $X$ as \defn{lifts of the clique $X$ to $G_r$}.
We remark that any lift $\hat X$ of a clique $X$ to $G_r$ is $r$-locally closed.
\cref{lem:CliqueAndLocalCover} ensures the following:
\begin{itemize} 
\item that every clique $\hat X$ of $G_r$ projects to a clique $X := p_r(\hat X)$ of $G$ and $\hat X$ is a lift of $X$ to $G_r$;
\item that  every clique of $G$ has a lift to $G_r$; and 
\item not only are two distinct lifts of any clique of $G$ to $G_r$ disjoint, but their closed neighborhoods are also disjoint.
\end{itemize}

We conclude this section by proving that statement \cref{basic:iii:sec} implies \cref{basic:i:sec} in \cref{thm:FIRSTTHREE}. 

\begin{proof}[Proof of \cref{thm:FIRSTTHREE}~\cref{basic:iii:sec}$\rightarrow$\cref{basic:i:sec}]
    \phantomsection\label{proof:FIRSTTHREE:iiiToi}
    Since the $r$-local cover $G_r$ is chordal by~\cref{basic:iii:sec}, its $r/2$-balls are chordal by \cref{lem:ChordalIsLocallyChordalForCliques}. Since the $r$-local covering $p_r$ preserves the $r/2$-balls, the $r/2$-balls in $G$ are also chordal. 
    Thus, $G$ is $r$-locally chordal, as desired by \cref{basic:i:sec}.
\end{proof}

\subsection{Background: Binary cycle space}

Let $G$ be a graph.
By \defn{$\cE(G)$}, we denote the $\Z_2$-vector space whose elements are the edge-subsets of $G$ and whose vector addition is the symmetric difference $\triangle$.
As is common in vector spaces, we say that an element $C$ of $\cE(G)$ is \defn{generated} by a subset $\cC$ of $\cE(G)$ if there is a finite $\cC' \subseteq \cC$ such that the symmetric difference of all the elements of~$\cC'$ is~$C$.
If all elements of a subset $\cD$ of $\cE(G)$ are generated by a set $\cC$, the set $\cC$ \defn{generates} $\cD$.
The \defn{binary cycle space} $\cZ(G)$ of $G$ is the subspace of $\cE(G)$ that consists of all elements generated by the set of (the edge sets of) cycles in $G$. 

An immediate consequence of the definition of the local covering is that the binary cycle space of any local cover is generated by its short cycles: 
\begin{lemma}[\cite{canonicalGD}, Lemma~4.6]\label{lem:BinaryCycleSpaceGenByShortCycles} 
    Let $G$ be a (possibly infinite) graph, and let $r \geq 0$ be an integer.
    Then the binary cycle space of the $r$-local cover $G_r$ of $G$ is generated by its cycles of length at most $r$ (equivalently: each of the cycles in the $r$-local cover $G_r$ is the symmetric difference of finitely many cycles of length at most $r$).
\end{lemma}

Recall that a \defn{triangle} in a graph $G$ is a cycle of length $3$ (equivalently: clique of size $3$) in $G$.
Any short cycle in an $r$-chordal graph is generated by a few triangles:

\begin{lemma}\label{lem:CyclesGenByTriangles}
    Let $r \geq 3$ be an integer. 
    Every cycle $O$ of length $\ell \leq r$ in a (possibly infinite) $r$-chordal graph~$G$ is the symmetric difference of $\ell - 2$ triangles of $G$ whose vertex sets are included in $V(O)$.
\end{lemma}

\begin{proof}
    Let $O$ be a cycle in $G$ of length $\ell$.
    We proceed by induction on $\ell$. If $O$ has length three, then $O$ is a triangle, so $O$ is the symmetric difference of $|O|-2 = 1$ triangles. 
    Now assume that $O$ has length $\ell > 3$ (and at most $r$) and that \cref{lem:CyclesGenByTriangles} holds for all cycles of length less than $\ell$.
    Since $G$ is $r$-chordal, $O$ has a chord $e$ in~$G$.
    Thus, $O + e$ contains two cycles $O_1$ and $O_2$ whose intersection is exactly $e$ and whose union covers $O$.
    In particular, $O_1$ and $O_2$ are both shorter than $O$, the sum $|O_1| + |O_2|$ of their lengths is $|O| + 2$, and their symmetric difference $O_1 \triangle O_2$ is $O$.
    Now, we may apply the inductive hypothesis to obtain sets $\cT_1$ of size $|O_1| - 2$ and $\cT_2$ of size $|O_2|-2$ triangles whose respective symmetric differences are $O_1$ and $O_2$, and whose vertex sets are included in the respective $V(O_i) \subseteq V(O)$.
    Note that $\cT_1$ and $\cT_2$ do not contain a common triangle, as the triangles in $\cT_i$ are contained in their respective $V(O_i)$ but $O_1$ and $O_2$ intersect only in the edge $e$.
    Hence, the symmetric difference of all triangles in $\cT \coloneqq \cT_1 \cup \cT_2 = \cT_1 \triangle \cT_2$ is $O_1 \triangle O_2= O$ and 
    \[
    |\cT| = |\cT_1| + |\cT_2| = (|O_1| -2) + (|O_2| - 2) = |O_1| + |O_2| - 4 = |O| - 2 = \ell -2.
    \]
    Therefore, $\cT$ is the desired set of $\ell -2$ triangles whose symmetric difference is~$O$.
\end{proof}

\subsection{Extra: Fundamental group of \texorpdfstring{$r$}{r}-chordal graphs}

In this section, we discuss some bonus observations about the fundamental group of locally chordal graphs, which are not directly related to later results in this paper.
We assume familiarity with the definition of the $r$-local covering via the fundamental group of a graph, the \emph{fundamental group} of a graph, its \emph{$r$-local subgroups}, and \emph{closed walks based at a vertex which stem from a cycle} from \cite[\S 4.1 \& 4.2]{canonicalGD}.
The fundamental group of a chordal graph $G$ is generated by closed walks around its triangles. More generally, we have the following:

\begin{lemma}\label{lem:FundGroupGenByTriangles}
    Let $r \geq 3$ be an integer and $x_0$ a vertex of a (possibly infinite) $r$-chordal graph $G$. 
    Every closed walk $W$ based at $x_0$ which stems from  a cycle $O$ of length $\ell \leq r$ is the concatenation of $\ell-2$ closed walks based at $x_0$ which stem from triangles whose vertex sets are included in $V(O)$.
\end{lemma}

\begin{proof}
    This can be shown by following the proof of \cref{lem:CyclesGenByTriangles} but additionally tracking the base walk of a closed stemming from a cycle $O$.
\end{proof}

Following the original definition of the $r$-local covering in \cite{canonicalGD}, an immediate corollary is that, for every $r \geq 3$, the $3$-local cover of an $r$-chordal graph is the $r$-local covering:

\begin{corollary}
    Let $G$ be a (possibly infinite) $r$-chordal graph and $r \geq 3$ an integer.
    Then the $3$-local covering $p_3$ is isomorphic to the $r$-local covering $p_r$ of $G$.
    In particular, if $G$ is chordal, then the $3$-local covering $p_3$ is isomorphic to the trivial covering $\id_G$ of $G$.
\end{corollary}

\begin{proof}
    As the characteristic subgroup $\pi_1^r(G)$ of $p_r$ is generated by the closed walks stemming from cycles of length at most $r$ and $G$ is $r$-chordal, \cref{lem:FundGroupGenByTriangles} yields that $\pi_1^r(G)$ is generated by the closed walks stemming from triangles, i.e.\ $\pi_1^3(G) = \pi_1^r(G)$.
    This shows the desired statement.

    The in-particular part is shown analogously by considering the fundamental group $\pi_1(G)$ itself instead of the $r$-local subgroup $\pi_1^r(G)$.
\end{proof}

\subsection{Algebraic characterization of (locally) chordal graphs}\label{subsec:binarycyclespace}

In this section, we prove the algebraic characterization of locally chordal graphs, which we restate here for convenience.

\algLocChordal*

\begin{proof}
    If $G$ is $r$-chordal, then \cref{lem:CyclesGenByTriangles} ensures that all its cycles of length at most $r$ are generated by triangles. 
    Conversely, assume that all cycles of $G$ of length at most $r$ are generated by triangles. Consider a cycle $O$ in $G$ of length at least $4$ and at most $r$.
    We show that $O$ has a chord.
    Since the cycles in $G$ of length at most $r$ are generated by triangles, there is a finite set of triangles whose symmetric difference is $O$.
    Fix $\cT$ to be a smallest such set. 
    If all three vertices of some triangle in $\cT$ are in $O$, then $O$ has a chord, as desired.
    Suppose for a contradiction that this is not the case.
    Let $T$ be a triangle in $\cT$.
    By assumption, there is a vertex $c$ of $T$ that is not on $O$. 

     Let $e_1, \hdots, e_q$ be the edges of $\bigcup_{T \in \cT} T$ incident with $c$, with endpoints $c$ and $v_1, \hdots, v_q$, respectively. Let~$H$ be the subgraph of $G$ with vertex set $\{v_1, \hdots, v_q\}$ and edge set consisting of edges $e$ with both endpoints in $\{v_1, \hdots, v_q\}$ such that $e$ is an edge of a triangle in $\cT$ containing $c$. 
     Since the symmetric difference of the triangles in $\cT$ is exactly $O$ and $c$ is not in $O$, it follows that every edge $e_1, \hdots, e_q$ disappears in the symmetric difference of $\cT$. Therefore, each edge $e_i$ is contained in at least two triangles of $\cT$. 
     Also every triangle in $\cT$ containing $e_i$ uses an edge from $H$. 
     Therefore, $H$ has minimum degree at least two, so $H$ contains a cycle~$C$.
    We may assume by reordering the $e_i$ and $v_i$ that $C = v_1v_2 \dots v_\ell v_1$.
    Then $C$ is the symmetric difference of the $\ell$ triangles $T_i$ formed by the edges $e_i,e_{i+1}, v_iv_{i+1}$, where $e_{\ell+1} = e_1$ and $v_{\ell +1}= v_1$.
    It follows from the  definition of $H$ that the $T_i$ are in $\cT$.

    If $G[V(C)]$ contains an induced cycle $C'$ of length at least four, then $G[V(C') \cup \{c\}]$ is a wheel with hub $c$, which is a contradiction.
    Hence, we may assume that $G[V(C)]$ is chordal. 
    Now, \cref{lem:CyclesGenByTriangles} yields a set $\cT^*$ of $\ell-2$ triangles in $G[V(C)]$ whose symmetric difference is $C$, and $\cT' \coloneqq (\cT \setminus \{T_1, \dots, T_\ell\}) \triangle \cT^*$ contradicts the choice of $\cT$, because $\cT'$ is smaller than $\cT$.
\end{proof}

We thank the user \href{https://math.stackexchange.com/users/1047163/kabel-abel}{kabel abel} on Mathematics Stack Exchange for their inspiring thoughts \cite{kabelabel} which we developed further to \cref{thm:chordalIFF-CSgenbyTriangles-No-Induced-wheels}. 
Since a graph is chordal if and only if it is wheel-free and $r$-chordal for every integer $r \geq 0$, an immediate corollary of \cref{thm:chordalIFF-CSgenbyTriangles-No-Induced-wheels} is the following algebraic characterization of chordal graphs, which was previously unknown.
It is an affirmative answer to the question asked by the author Jonas Kobler that was left unanswered in their discussion.

\algChordal*

As another consequence of \cref{thm:chordalIFF-CSgenbyTriangles-No-Induced-wheels}, we obtain the following characterization for the $r$-local cover of any graph being chordal:

\begin{theorem}\label{thm:CharOfLocCoverChordal}
    Let $G$ be a (possibly infinite) graph, and let $r \geq 3$ be an integer.
    Then the $r$-local cover $G_r$ of $G$ is chordal if and only if $G_r$ is $r$-chordal and wheel-free.
\end{theorem}

\begin{proof}
    As every induced subgraph of a chordal graph is chordal, the $r$-local cover $G_r$ is $r$-chordal and wheel-free, if it is chordal.
    Conversely, assume that the $r$-local cover $G_r$ is $r$-chordal and wheel-free.
    By \cref{lem:BinaryCycleSpaceGenByShortCycles}, the binary cycle space of the $r$-local cover $G_r$ is generated by its cycles of length at most $r$.
    \cref{lem:CyclesGenByTriangles} yields that any cycle of length at most $r$ in the $r$-chordal graph $G_r$ is generated by its triangles. 
    All in all, the binary cycle space of $G_r$ is generated by its triangles.
    Since $G_r$ is also wheel-free by assumption, it follows from \cref{thm:chordalIFF-CSgenbyTriangles-No-Induced-wheels} that $G_r$ is chordal.
\end{proof}

From \cref{thm:CharOfLocCoverChordal} we derive that statement \cref{basic:ii:sec} implies \cref{basic:iii:sec} in \cref{thm:FIRSTTHREE}:

\begin{proof}[Proof of \cref{thm:FIRSTTHREE}~\cref{basic:ii:sec}$\rightarrow$\cref{basic:iii:sec}]
    \phantomsection\label{proof:FIRSTTHREE:iiToiii}
    Assume that \cref{basic:ii:sec} holds, i.e.\ $G$ is $r$-chordal and wheel-free.
    We aim to show that \cref{basic:iii:sec} holds, i.e.\ $G_r$ is chordal. 
    By \cref{thm:CharOfLocCoverChordal}, it suffices to show that $G_r$ is also $r$-chordal and wheel-free.
    Indeed, any cycle in $G_r$ of length at most $r$ is contained in the $r/2$-ball around each of its vertices and any wheel in $G_r$ is contained in the $3/2$-ball around its hub.
    Since $p_r$ is $r/2$-ball-preserving, whenever these graphs occur as induced subgraphs in $G_r$, they would also occur in the corresponding $r/2$-ball in $G$ and thus in $G$ itself as induced subgraph.
    But $G$ is $r$-chordal and wheel-free by assumption.
    Hence, $G_r$ is also $r$-chordal and wheel-free, as desired.
\end{proof}

For the interested reader, we remark that one can prove that statement \cref{basic:ii} implies \cref{basic:i} in \cref{thm:FIRSTTHREE} directly, avoiding the use of the $r$-local cover, analogous to the above proof of \hyperref[proof:FIRSTTHREE:iiToiii]{\cref*{basic:ii}$\rightarrow$\cref*{basic:iii}}:
instead of applying \cref{lem:BinaryCycleSpaceGenByShortCycles}, one may use that the binary cycle space of any $r/2$-ball is generated by its cycles of length at most $r$.
This is immediate from the fact that, for any vertex $v$ of $G$, the fundamental cycles of a breadth-first search (BFS) tree of $B_{r/2}(v)$ rooted at $v$ have length at most $r$, since it is folklore that the fundamental cycles of any spanning tree of a graph generate its binary cycle space (e.g. \cite[Theorem~1.9.5~(i)]{bibel}).

\section{Characterization via minimal local separators}\label{sec:CharViaMinLocSep}

In this section, we prove the characterization of locally chordal graphs via minimal local separators, i.e.\ the equivalence of statements \cref{basic:i} and \cref{basic:iv} of \cref{BasicCharacterization}, which we restate here for convenience: 

\begin{customthm}{\cref*{BasicCharacterization}}\label{conj:local-separators}
    Let $G$ be a (possibly infinite) graph and let $r \geq 3$ be an integer. The following are equivalent:
    \begin{enumerate}
        \item \label{basic:i:third} $G$ is $r$-locally chordal.
        \setcounter{enumi}{3}
        \item \label{basic:iv:third}
        Every minimal $r$-local separator of $G$ is a clique.
    \end{enumerate}
\end{customthm}

\noindent This extends Dirac's characterization of chordal graphs via minimal separators (\cref{thm:CharacterisationChordalViaMinimalVertexSeparator}) to locally chordal graphs.

\subsection{Background: Local separators \texorpdfstring{\&}{and} components}
Let us explain the local analog of a separator introduced by Carmesin, Jacobs, Kurkofka and Knappe~\cite{computelocalSeps}.
For this, we first need to introduce their local analog of components.
Let $r \geq 0$ be an integer, and let $X$ be a vertex-subset of a graph $G$.
An \defn{$X$-walk} is a walk whose first and last vertices are in $X$ and which otherwise avoids $X$.
A walk in $G$ is \defn{$r$-local} if it is a walk in a cycle of length at most $r$ of $G$ or a walk traversing a single edge.
By \defn{$\partial X$}, we denote the \defn{edge boundary of $X$} in $G$, i.e.\ the set of all edges in $G$ with precisely one end vertex in $X$.
The \defn{$r$-local components at~$X$} in $G$ are the equivalence classes of the transitive closure of the relation $\sim_r$ on $\partial X$ defined by letting
\begin{center}
    \defn{$e\sim_r f$ at $X$ in $G$} $:$ $e$ and $f$ are the first and last edge on an $r$-local $X$-walk in~$G$, or $e=f$.
\end{center}
We remark that for $r \coloneqq \infty$, the $\infty$-local components at $X$ precisely correspond to the \defn{componental cuts at $X$}, i.e.\ the sets $E_G(C,X)$ consisting of the edges between $X$ and a component $C$ of $G-X$.
We say that $X$ is an \defn{$r$-local separator} of $G$ if there are at least two $r$-local components at $X$.

We extend the definitions regarding local separators from \cite{computelocalSeps} by describing what it means for two vertices to be ``locally separated'' and introduce ``minimal local separators.''
Let $u,w$ be two vertices of a graph $G$ of distance $2$.
Fix a vertex $v$ of $G$ with $uv,vw \in E(G)$.
A vertex-subset $X$ of $G$ \defn{$r$-locally separates $u$ and $w$} in $G$ with respect to $v$, and $X$ is an \defn{$r$-local $u$--$w$ separator} of $G$ with respect to $v$, if $X$ avoids $
\{u,w\}$, contains $v$, and $uv$ and $vw$ are in distinct $r$-local components at $X$.
A vertex-subset $X$ of $G$ is a \defn{minimal $r$-local separator} of $G$ if there are two vertices $u,w$ of distance $2$ in $G$ such that $X$ is an inclusion-minimal $r$-local $u$--$w$ separator of $G$ with respect to some vertex $v$ with $uv,vw \in E(G)$.

Note that if $r \geq 4$, then these definitions are independent of the choice of the vertex $v$ which witnesses that the distance between $u$ and $w$ is $2$.
We remark that $X \coloneqq V(G) \setminus \{u,w\}$ always forms an $r$-local $u$--$w$-separator of $G$ (with respect to any suitable $v$).

Every local $u$--$w$ separator with respect to $v$ is a (global) $u$--$w$ separator in the ball around $v$:

\begin{lemma}\label{lem:LocalSeparatorLocallySeparates}
    Let $G$ be a graph and $r \geq 0$ an integer. 
    Let $X$ be a vertex-subset of $G$ that $r$-locally separates two vertices $u$ and $w$ with respect to another vertex $v$ with $uv,vw \in E(G)$ in $G$.
    Then $X \cap N^{r/2}[v]$ is a $u$--$w$ separator of $B_{r/2}(v)$.
\end{lemma}

\begin{proof}
    Suppose for a contradiction that there is a $u$--$w$ path $P = v_0e_1v_1 \dots e_\ell v_\ell$ in $B_{r/2}(v)$ avoiding $X$.
    Fix a BFS tree $T$ of $B_{r/2}(v)$ rooted at $v$ that contains the edges $uv$ and $vw$.
    Let $i \in [\ell]$.
    If $e_i$ is in $T$, then let both $f_i^-, f_i^+$ be the last edge on the (unique) inclusion-minimal path in $T$ starting at $v$ and containing $e_i$.
    Otherwise,
    let $W_i$ be the fundamental closed walk based at $v$ through $e_i$ with respect to $T$ that traverses $e_i$ from $v_{i-1}$ to $v_i$.\footnote{This is the unique reduced closed walk based at $v$ which is entirely in $T+e_i$ but it traverses $e_i$ precisely once and that is from $v_{i-1}$ to $v_i$; we refer the reader to \cite[\S 4.1]{canonicalGD} for more details.}
    Let $f^-_i$ be the last edge on $W_i$ before $e_i$ which is in $\partial X$.
    Let $f^+_i$ be the first edge on $W_i$ after $e_i$ which is in $\partial X$.
    Note that $f^-_{i+1} = f^+_i$ for $i \in [\ell - 1]$.
    Since $T$ contains $uv,vw$, we have $f^-_1 = uv$ and $f^+_{\ell} = vw$.
    As $T$ is a BFS tree of $B_{r/2}(v)$ rooted at $v$, the $W_i$ have length at most $r$.
    Thus, either $f_i^- = f_i^+$ or the subwalk of $W_i$ that starts with $f_i^-$ and ends with $f_i^+$ is an $r$-local $X$-walk.
    Hence, $f^-_i \sim_r f^+_i$ at $X$ for every $i \in [\ell]$.
    All in all, $f^-_1 = uv$ and $f^+_{\ell} = vw$ are in the same $r$-local component at $X$, which contradicts that $X$ $r$-locally separates $u$ and $w$.
\end{proof}

\subsection{Proof of the characterization via minimal local separators}

Now we show \hyperref[proof:FIRSTTHREE:iiiToiv]{\cref*{BasicCharacterization}~\cref*{basic:iii}$\rightarrow$\cref*{basic:iv}} and \hyperref[proof:FIRSTTHREE:ivToii]{\cref*{BasicCharacterization}~\cref*{basic:iv}$\rightarrow$\cref*{basic:ii}}.
Together with the equivalence of the first three statements \cref{basic:i}, \cref{basic:ii} and \cref{basic:iii} of \cref{BasicCharacterization}, which we have shown in the previous section, this yields the desired characterization via minimal local separators, \cref{conj:local-separators}~\cref{basic:i:third}$\leftrightarrow$\cref{basic:iv:third}.

\begin{proof}[Proof of \cref{BasicCharacterization}~\cref{basic:iv}$\rightarrow$\cref{basic:ii}]
\phantomsection\label{proof:FIRSTTHREE:ivToii}
    Assume that \cref{basic:iv} holds.
    We claim that $G$ contains neither a wheel $W_n$ with $n \geq 4$ nor a cycle of length at least $4$ and at most $r$ as an induced subgraph.
    Suppose for a contradiction that $G$ contains a cycle $O$ of length at least $4$ but at most $r$ (analogously: contains a wheel $W_n$ with $n \geq 4$ and whose rim we denote by $O$) as an induced subgraph.
    Let $u,v,w$ be three consecutive vertices on $O$ (analogously: let $u,w$ be two non-adjacent vertices on the rim $O$ and let $v$ be the hub of wheel). 
    In particular, $u,w$ have distance $2$ and $uv,vw \in E(G)$.
    Let $X$ be an inclusion-minimal $r$-local $u$--$w$ separator in $G$.
    Since $uv$ and $vw$ are in distinct $r$-local components at $X$, the short cycle $O$ meets $X$ in at least one vertex $x$ other than $v$ (analogously: each $P_i$ of the two $u$--$w$ paths $P_1,P_2$ in the rim $O$ meets $X$ in at least one vertex $x_i$). 
    Thus, $X$ is a minimal $r$-local separator in $G$, which contains two non-adjacent vertices $v,x$ (analogously: $x_1,x_2$) of $G$.
    This is a contradiction, since $X$ is a clique by \cref{basic:iv}.
\end{proof}

\begin{proof}[Proof of \cref{BasicCharacterization}~\cref{basic:iii}$\rightarrow$\cref{basic:iv}]
\phantomsection\label{proof:FIRSTTHREE:iiiToiv}

    Assume that \cref{basic:iii} holds.
    Let $X$ be a minimal $r$-local $u$--$w$ separator in $G$ with respect to a vertex $v$ with $uv,vw \in E(G)$.
    By \cref{lem:LocalSeparatorLocallySeparates}, $X \cap N^{r/2}[v]$ is a $u$--$w$ separator of $B_{r/2}(v)$.

    Let $Y$ be a minimal $u$--$w$ separator of $B_{r/2}(v)$.
    Since $G$ is $r$-locally chordal, the ball $B_{r/2}(v)$ is chordal, so $Y$ is a clique by \cref{thm:CharacterisationChordalViaMinimalVertexSeparator}.
    Thus, it suffices to show that $Y$ is an $r$-local $u$--$w$ separator in $G$ with respect to $v$: the minimality of $X$ then yields that $X \cap N^{r/2}[v] = Y$, which is a clique.
    
    Let $\hat Y$ be a lift of the clique $Y$ to $G_r$, and let $\hat u$ and $\hat w$ be the respective lifts of $u$ and $w$ incident to the unique lift $\hat v$ of $v$ in $\hat Y$.
    Since $\hat Y$ is an $r$-locally closed lift of the vertex-subset $Y$ to $G_r$, it follows from \cite[Projection~4.16]{computelocalSeps} that it remains to show that $\hat Y$ is an $r$-local $\hat u$--$\hat w$ separator of $G_r$.
    For this, it suffices to ensure that $\hat Y$ is an $\hat u$--$\hat w$ separator of $B_{G_r}(\hat Y,r/2)$.
    
    So suppose for a contradiction that there is a $\hat u$--$\hat w$ path $\hat P$ in $B_{G_r}(\hat Y,r/2)$ avoiding $\hat Y$.
    Since $p_r$ preserves $r/2$-balls, the vertices $\hat u$ and $\hat v$ live in distinct components of $B_{G_r}(\hat v,r/2) - \hat Y$.
    Thus, $\hat P$ contains a $B_{G_r}(\hat v,r/2)$-path through $N^{r/2}_{G_r}[\hat Y] \setminus N^{r/2}_{G_r}[\hat v]$ that joins two distinct components of $B_{G_r}(\hat v,r/2) - \hat Y$.
    Let $\hat Q$ be a shortest such path, and let $\hat R$ be a shortest path through $N^{r/2-1}_{G_r}[\hat v]$ joining the endvertices $\hat x, \hat y$ of $\hat Q$.
    Since the vertices $\hat x$ and $\hat y$ are in distinct components of $B_{G_r}(v, r/2) - \hat Y$, there is no edge between $\hat x$ and $\hat y$ in $G_r$.
    Hence, $\hat Q \cup \hat R$ forms an induced cycle of length at least $4$, which contradicts that $B_{G_r}(\hat Y, r/2)$ is chordal by \cref{lem:ChordalIsLocallyChordalForCliques} and \cref{basic:iii}.

\end{proof}

\subsection{Comparison: Minimal vs.\ tight local separators}

We conclude by highlighting a difference between minimal local separators and minimal separators. 
A vertex-subset $X$ of a graph $G$ is a \defn{tight} separator of $G$ if there are at least two components~$C$ of~$G-X$ are \defn{full}, i.e.\ $N_G(C) = X$.\footnote{Full components are sometimes also called {\em tight components}.}
It is easy to see that a separator $X$ of $G$ is tight if and only if it is minimal.

In \cite{computelocalSeps}, the authors also introduced a local analog of tight separator.
To define it here, we first need to introduce their local analog of full components.
An $r$-local component~$F$ at $X$ is \defn{full} if every vertex in~$X$ is incident to some edge in~$F$.
Moreover, an $r$-local separator $X$ of $G$ is \defn{tight} if there are at least two full $r$-local components at $X$ in $G$.
It is easy to see that every minimal local separator is tight, but in contrast to the case of global separators, not every tight local separator is minimal.
Also, not every tight local separator of a locally chordal graph is a clique:

\begin{example}\label{example:TightNotClique}
    The graph depicted in \cref{fig:TightNotClique} is $r$-locally chordal graph for $r=3$ and has a tight $r$-local separator that is not a clique.
\end{example}

\begin{figure}[ht]
    \centering
    \includegraphics[width=0.2\textwidth]{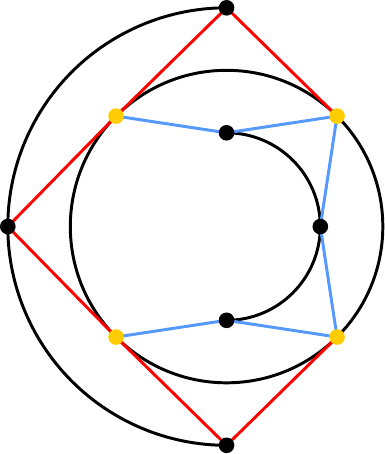}
    \caption{
    A tight $r$-local separator (yellow) in an $r$-locally chordal graph that is not a clique for $r = 3$.}
    \label{fig:TightNotClique}
\end{figure}

\begin{proof}
    Let $G$ be the graph depicted in \cref{fig:TightNotClique}.
    Each $3/2$-ball of $G$ is chordal, so $G$ is $3$-locally chordal.
    The red and blue edges, respectively, form full $3$-local components at the set $X$ of yellow vertices, but $X$ is not a clique in $G$.
\end{proof}

\section{Future work}

We can define ``local'' versions of any graph class in the same way that we defined locally chordal graphs: for a graph class $\mathcal{C}$, a graph $G$ is \defn{$r$-locally $\mathcal{C}$} if $B_G(v, r/2) \in \mathcal{C}$ for every vertex $v$ of $G$. 
Note that $r$-locally chordal is precisely $r$-locally $\cC$ for the class $\cC$ of chordal graphs.
Which other local versions of graph classes have nice characterizations? We are especially interested in when local versions of graph classes can be characterized by the local cover:

\begin{question}
    Which graph classes $\cC$ satisfy that a graph $G$ is $r$-locally $\cC$ if and only if its $r$-local cover $G_r$ is $\cC$?
\end{question}
Observe that if a graph class $\cC$ is closed under taking $r/2$-balls, i.e. for all $G \in \cC$ and $v \in G$ the $r/2$-balls $B_G(v, r/2)$ are in $\cC$, then $G_r \in \cC$ implies that $G$ is $r$-locally $\cC$.

One of the most common uses of ``local'' in the graph theory literature is to call graphs with no short cycles ``locally tree-like.'' This turns out to be a simple example of a local class with nice characterizations. A graph is \defn{$r$-locally acyclic} if each of its $r/2$-balls is acyclic (equivalently: a tree). The following is immediate from the definitions: 
\begin{theorem}
    Let $G$ be a graph and $r \geq 0$ an integer. The following are equivalent. 
    \begin{enumerate}
        \item $G$ is $r$-locally acyclic. 
        \item $G$ has girth greater than $r$. 
        \item $G_r$ is acyclic (equivalently: $G_r$ is a forest). 
    \end{enumerate}
\end{theorem}

Both chordal graphs and acyclic graphs have neat descriptions in terms of tree-decompositions: chordal graphs are precisely the graphs with tree-decompositions into cliques, and acyclic graphs are precisely the graphs of treewidth at most $1$. 
Graphs of treewidth at most $2$ also have a nice structural characterization. A graph is \defn{series-parallel} if it does not contain a $K_4$ minor. A graph is series-parallel if and only if it has treewidth at most 2. A graph is \defn{$r$-locally series-parallel} if each of its $r/2$-balls is series-parallel. 

Observe that $K_4$ is very similar to a wheel: it is a vertex $v$ complete to a cycle $C_3$. An \defn{$r$-local subdivision} of a wheel or of a $K_4$ is a subdivision $H$ of a wheel or a $K_4$ such that $H$ is contained in $B_H(v, r/2)$ for some vertex $v$ of $H$. We conjecture that locally series-parallel graphs admit a characterization similar to that of locally chordal graphs and locally acyclic graphs: 

\begin{conjecture}\label{conj:seriesparallel}
    Let $G$ be a (possibly infinite) graph and let $r \geq 3$ be an integer.
    Then the following are equivalent:
    \begin{enumerate}
        \item $G$ is $r$-locally series-parallel.
        \item $G$ does not contain an $r$-local subdivision of a wheel or of a $K_4$.
        \item $G_r$ is series-parallel.
    \end{enumerate}
\end{conjecture}

 We aim to prove \cref{conj:seriesparallel} in a future work \cite{LocallySeriesParallel}.

\printbibliography

\end{document}